\documentclass[12pt,a4paper,reqno]{amsart}
\usepackage[mathletters]{ucs}
\usepackage[utf8x]{inputenc}
\usepackage{amsthm, amssymb, amsmath}
\usepackage{xcolor, graphicx}
\usepackage{wasysym} 
\usepackage{amsfonts} 
\usepackage[margin = 2cm]{geometry}
\usepackage{hyperref}

\usepackage{soul}

\newtheorem{theorem}{Theorem}
\newtheorem{lemma}[theorem]{Lemma}
\newtheorem{proposition}[theorem]{Proposition} 
\newtheorem{rem}[theorem]{Remark}

\newtheorem{assumption}{Assumption}
\newtheorem{corollary}[theorem]{Corollary}

\theoremstyle{definition}
\newtheorem{definition}[theorem]{Definition}

\DeclareMathOperator{\dist}{dist}
\DeclareMathOperator{\diam}{diam}

\def \MidEdgesExcept {\mathcal E^{\mathrm{mid}}_{\circ}}
\def \HalfEdges {\mathcal E^{\mathrm{half}}}
\def \Faces {\mathcal F}

\def \Blue{\mathbf{blue}}
\def \Yellow{\mathbf{yellow}}

\def \Pperc {\mathbb P^{\mathrm{perc}}}
\def \Ploop {\mathbb P^{\mathrm{loop}}}
\def \Eloop {\mathbb E^{\mathrm{loop}}}

\def\T{{\mathbb T}}
\def\LIM#1{#1^{\bullet}}

\def\IP{\mathcal{IP}} 
\def\CP#1#2{[{ #1 \leftrightsquigarrow #2}]} 
\def\CPP#1#2#3#4{[{ #1 \leftrightsquigarrow #2, #3 \leftrightsquigarrow #4}]} 

\title{Percolation and $O(1)$ loop model}
\author{Mikhail Khristoforov}
\address{St. Petersburg State University, St. Petersburg, Russia}
\email{Mikhail.Khristoforov@gmail.com}
\author{Stanislav Smirnov}
\address{Université de Genève, Genève 4, Switzerland and St. Petersburg State University, St. Petersburg, Russia}
\email{Stanislav.Smirnov@unige.ch}

\begin{document}

\begin{abstract}
We present an ``ultimate'' proof of Cardy's formula for the critical percolation on the hexagonal lattice
\cite{Smirnov01criticalpercolation}, showing the existence of the universal and conformally invariant scaling limit of crossing
probabilities.
The new approach is more conceptual, less technically demanding, and is amenable to generalizations.
\end{abstract}

\maketitle

\section{Introduction} 

Percolation was introduced by Broadbent and Hammersley \cite{BroadbentHammersley57} to model how a fluid spreads through a
random medium.
It is very easy to define: sites (or bonds) of a graph are declared open or closed independently (in Bernoulli percolation)
with probabilities $p$ and $1-p$ correspondingly, and connected open clusters are studied.
Nevertheless, this percolation model exhibits a very rich and complicated behavior even on planar lattices, including a phase
transition at some lattice-dependent value $p_c$.

In particular, the ``crossing probability'' (of the existence of an open cluster connecting two opposite sides of a fixed
shape), as the mesh of the lattice tends to 0, tends to 0 when $p<p_c$ and tends to 1 when $p>p_c$ --- a ``sharp threshold
phenomenon''.

Meanwhile, for regular lattices, the Russo-Seymour-Welsh a priori estimates guarantee that for $p=p_c$ the ``crossing
probability'' stays bounded away from 0 and 1, strongly suggesting the existence of a non-trivial ``scaling limit''.

In  the seminal work \cite{LanglandsEtAl94} Langlands, Pouliot, and Saint-Aubin conducted a number of computer experiments
suggesting that there is a universal (lattice-independent) scaling limit of the crossing probabilities at criticality which is
furthermore conformally invariant, i.e. depends only on the conformal modulus of the quadrangular shape. 

Almost immediately  Cardy \cite{Cardy92} derived (unrigorously) the exact formula for the limit as a hypergeometric function of
the modulus, which Carleson observed to take a particularly nice form for an equilateral triangle with one more marked point on
a side.

In 2000 the second author provided a rigorous proof of the Cardy's prediction for the critical percolation on the triangular
lattice, which allowed to deduce many of its properties.

This proof has never appeared in a journal form not in the least because we felt it somehow artificial and having unexplained
complications, albeit still elegant. The result was widely used to deduce various properties of percolation, such as the
convergence of interfaces to $SLE_6$ and exact values of the critical exponents. It also stimulated an extremely fruitful
approach to study models by tools of discrete holomorphic or harmonic observables \cite{LawlerSchrammWerner04, CD-CHKS,
Hugo13Parafermionic}. 

It took some time to arrive at what we think is ``the proof from the Book'', which we present in this article.
On one hand, the new proof is more ``ideologically  fruitful'', while it can be literally translated into the old one; the
objects under consideration are classical disorder operators, rather than some curiosities of uncertain origins. The
parafermionic nature of the observable and its relation to similar objects in the Ising and other models becomes  clear, cf.
\cite{Smirnov10, ChelkakSmirnov12, HonglerSmirnov13}.
On the other hand, the proof is much more straightforward. In particular, discrete holomorphicity becomes exact and 
there is no need to estimate errorterms.

Moreover, the new description of the observable admits immediate generalizations allowing one to obtain several results (e.g.
Schramm's formula \cite{Schramm01PercolationFormula} or formulae for the probabilities of the link patterns in the
topological hexagon \cite{FloresSimmonsKleban15}) in the spirit of this article. 
We intend to show that in the subsequent papers \cite{tmpKSS, tmpKK}.

Justification of Cardy's formula for graphs other than the hexagonal lattice remains an open problem and we have some hope that
the new point of view  could become useful there.

{\it Acknowledgments:}
The authors are  grateful to Hugo Duminil-Copin, Dmitry Krachun, Ioan Manolescu and Mikhail Skopenkov for fruitful discussions.
The work is supported by the Swiss NSF and ERC Advanced Grants 340340 and  741487.
The section \ref{Section:Spinor} was written entirely under support of Russian Science Foundation grant 19-71-30002.

\subsection{Percolation model}
We will study critical site percolation on triangular lattice, or equivalently plaquette percolation on hexagonal lattice.
Let $\mathbb C_{\hexagon}^{δ}$ be a hexagonal lattice of mesh size $δ$ on $\mathbb C$. A $\hexagon^δ$-domain (hexagonal domain)
is a bounded simply-connected domain glued from the faces of $\mathbb C_{\hexagon}^{δ}$ and a $\hexagon$-domain is a domain
that is $\hexagon^δ$-domain for some $δ$. 
By  $\Faces(Ω)$ and $\HalfEdges (Ω)$ we denote  the sets of faces and  half-edges of a $\hexagon$-domain $Ω$ respectively.

The percolation model on $Ω$ is the uniform measure on the set of all $2^{\#\Faces(Ω)}$ colorings of faces of $Ω$ in two
colors, say blue and yellow, we denote this measure by $\Pperc_{Ω}$. 
For a given coloring 
$$σ \colon \Faces(Ω) \to \{\Yellow,\Blue\}$$ 
if there is a $σ$-blue path between two sets $X$ and $Y$, we say that $X$ and $Y$  are {\it connected} and write $X
\leftrightarrow Y$.

The scaling limits of probabilities to be connected in the percolation model are proven to exist and be conformally invariant.
In this article we give a revised proof of the fundamental result in the area.

\begin{theorem}[Smirnov'01, \cite{Smirnov01criticalpercolation}]
\label{Thm1}
  If  $\{(Ω^δ, A^δ, B^δ, C^δ, D^δ) \}_δ$ approaches $(\LIM  {Ω}, \LIM A, \LIM B, \LIM C, \LIM D)$ (in the sense of 
  Definition \ref{definition:NiceApproximation}) then
  \begin{equation}
  \label{eq:thm}
    \lim\limits_{δ\searrow 0}\Pperc_{Ω^δ}
    [∂_{A^δ B^δ} Ω^δ ↔  ∂_{C^δ D^δ} Ω^δ] 
    = \frac {φ(\LIM C)- φ(\LIM D)} {φ(\LIM C) - φ(\LIM A)},
  \end{equation}
where $φ$ is the conformal map from $\LIM  {Ω}$ to an equilateral triangle, mapping $\LIM A, \LIM B, \LIM C$ to vertices.  
\end{theorem}

\subsection{Loop representation}
For a collection of half-edges $ξ\subset \HalfEdges(Ω)$ we denote by $∂ξ$ the set of vertices and mid-edges of $Ω$ that are
adjacent to an odd number of half-edges of $ξ$. 

Let $U = \{u_1, \dots u_k\}$ be a set of $k$ mid-edges of $Ω$, we call them {\it marked points}.
We define
$$ W_Ω(u_1, \dots u_k) := W_Ω(U) := 
\left \{ ξ \subset  \HalfEdges(Ω) : ∂ ξ = U  \right \}$$
and call elements of $ W_Ω(u_1, \dots u_k)$ {\it loop configurations with disorders at marked points}.
Assume that $k$ is even, then this set in non-empty. 

Let $ξ$ be such a loop configuration. 
The union of half-edges of $ξ$ still will be denoted by $ξ$.
The union of connectivity components of $ξ$ containing at least one marked point we denote by $\IP(ξ)$ and call the {\it Interface Part} of $ξ$.
Note that $ξ\setminus \IP(ξ)$ is a union of disjoint loops and $\IP(ξ)$ is a union of disjoint paths, matching marked points.
This matching is called {\it Link Pattern} of $ξ$.

By $\Ploop_{Ω, U}$ we denote the uniform measure on $W_Ω(U)$.
Note that $\Ploop_{Ω, \emptyset}$ corresponds to the loop $O(1)$ model (or, equivalently,  the Ising model at the infinite temperature).
The matter of our interest is the law of the link pattern of the uniformly random loop configuration with disorders at marked points. 
Note that if $ξ_1$ and $ξ_2$ are loop configurations with the same disorders, then the symmetric difference $ξ_1 \oplus ξ_2$ is a union of loops. 
This implies that there are exactly $2^{\# \Faces (Ω)}$ loop configurations with given disorders.

If $z$ and $w$ are two points on $∂ Ω$ we denote by $∂_{zw}Ω$ the counterclockwise arc of $∂ Ω$ from $z$ to $w$. 
When $u_1, \dots, u_m$ are defined as points lying on the boundary of $Ω$ we always mean that they go in the counterclockwise
order and are indexed cyclically: $u_{n\pm m} := u_n$. 
For $j, j' \in \mathbb Z$  we use shorthands $∂_{jz} Ω := ∂_{u_jz} Ω$, $∂_{zj} Ω := ∂_{zu_j} Ω$, $∂_{jj'} Ω := ∂_{u_ju_{j'}} Ω$.
Additionally, if $m = 2l + 1$ is odd then $∂_j Ω := ∂_{u_{j+l}u_{j-l}} Ω$.

\begin{lemma}
\label{lemma:percloop}
 Let $u_1, \dots, u_4$ be four distinct mid-edges on $∂ Ω$. 
 Then there are two possible link patterns of a loop configuration $ξ$: 
 either $u_1$ is linked to $u_2$ and $u_3$ to $u_4$ in $\IP(ξ)$ or $u_1$ is linked to $u_4$ and $u_2$ to $u_3$. 
 We denote the corresponding events by $\CPP{u_1}{u_2}{u_3}{u_4}$ and $\CPP{u_1}{u_4}{u_2}{u_3}$. Then
$$\Pperc_Ω[∂_{u_1u_2} Ω ↔  ∂_{u_3u_4} Ω ] = 
 \Ploop_{Ω, \{u_1, u_2, u_3, u_4\}} \CPP{u_1}{u_4}{u_2}{u_3}.$$
\end{lemma}

\begin{proof}
For a coloring $σ$ one can construct a loop configuration $ξ = ξ(σ)$ with disorders at $u_1, \dots u_4$ by the following rule:
a half-edge $e$ belongs to $ξ(σ)$ if and only if the colors on the left and on the right of $e$ differ, (see Figure
\ref{fig:tennis}).
Here we assume that the outer boundary is blue along $∂_{12}Ω$ and  $∂_{34}Ω$ and is yellow  along $∂_{23}Ω$ and  $∂_{41}Ω$.
This map is a bijection between colorings and loop configurations with disorders at $u_1, \dots u_4$, moreover $∂_{u_1u_2} Ω ↔
∂_{u_3u_4} Ω$ in $σ$ if and only if $ \CPP{u_1}{u_4}{u_2}{u_3}$ in $ξ(σ)$.
\end{proof}

\begin{figure}[ht!]
   \centering
    \includegraphics[width=0.4\textwidth]{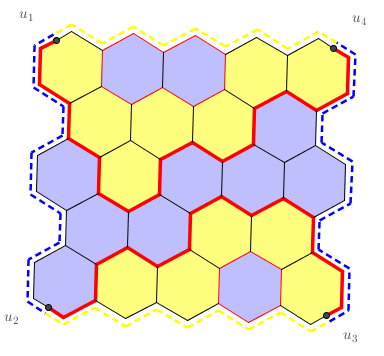}%
    \caption{Here $σ$ is drawn in blue and yellow and $ξ$ in red; $\IP(ξ)$ is thick and outer boundaries are dashed.}
    \label{fig:tennis}
\end{figure}

\subsection{Spinor percolation model}
\label{Section:Spinor}
Lemma \ref{lemma:percloop} shows the correspondence between loop configurations with disorders on the boundary and colorings in
two colors. One can naturally generalize this correspondence for the case when the disorders are allowed to lie inside the
domain. Indeed, let $u_1, \dots u_k$ be mid-edges of $Ω$, and let $ρ\colon \tilde{Ω}_{u_1, \dots u_k} \to Ω$ be the double
covering of $Ω$ ramified at  each $u_j$, so $\tilde{Ω}_{u_1, \dots u_k}$ includes two copies of each face of ${Ω}$. 
A {\it spinor coloring} is a map 
$$σ \colon \Faces(\tilde{Ω}_{u_1, \dots u_k}) \to \{\Yellow, \Blue\}$$
such that two $ρ$-preimages of any face of ${Ω}$ have different colors. 
Note that if each $u_j$ lies on the boundary then
$\tilde{Ω}_{u_1, \dots u_k}$ 
has the same structure of faces and mid-edges as the disjoint union of two copies of ${Ω}$.  
If $σ$ is a spinor coloring and $\tilde  {ξ}(σ)$ is the set of half-edges such that $σ$-colors on the left and the right of it
differ, then $ξ = ρ(\tilde {ξ} (σ))$ is a loop configuration with disorders at  $u_1, \dots u_k$, the vice-versa is also true.

The spinor percolation model is the uniform measure on the set of all spinor colorings. There are several immediate advantages
of working with it. In particular, the interfaces can be sampled by the standard revealment process (and those processes can be
naturally coupled for models on the same domain with different disorders until the moment when the interface `disconnects
disorders').

\section{Discrete holomorphicity} 
 Let $u_1, u_2, u_3$ be three distinct mid-edges lying in the counterclockwise order on $∂ Ω$ and $z$ be any mid-edge
 distinct from them. There are three possible link patterns for a loop configuration with disorders at $u_1, u_2, u_3,
 z$. If $z$ is connected to $u_j$ and $u_{j-1}$ is connected to $u_{j+1}$ by the edges of $\IP(ξ)$ we say that the event $\CP
 {z}{u_j} = \CPP {z}{u_j}{u_{j-1}}{u_{j+1}}$ occurs.

\begin{figure}[ht!]
    \centering
    \includegraphics[width=0.2\textwidth]{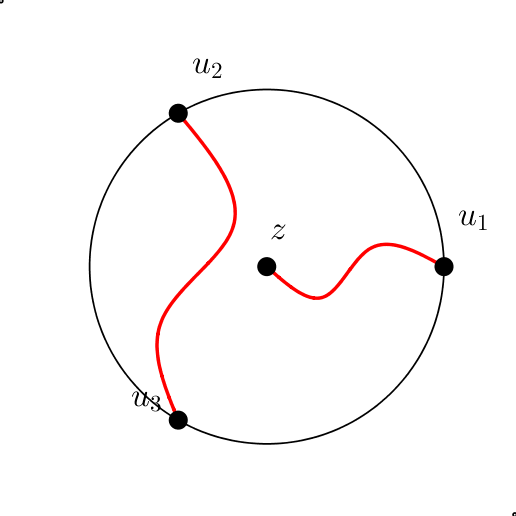}%
    \includegraphics[width=0.2\textwidth]{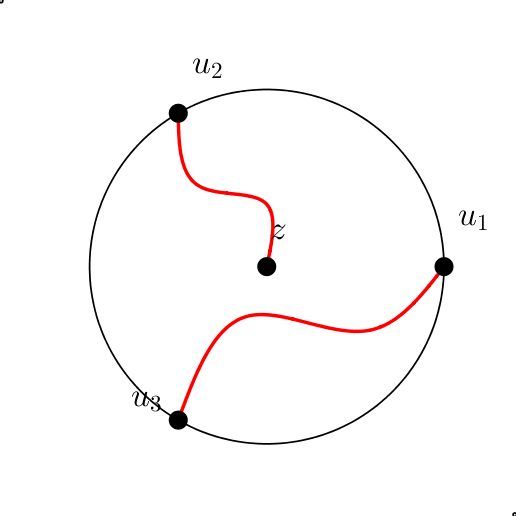}%
    \includegraphics[width=0.2\textwidth]{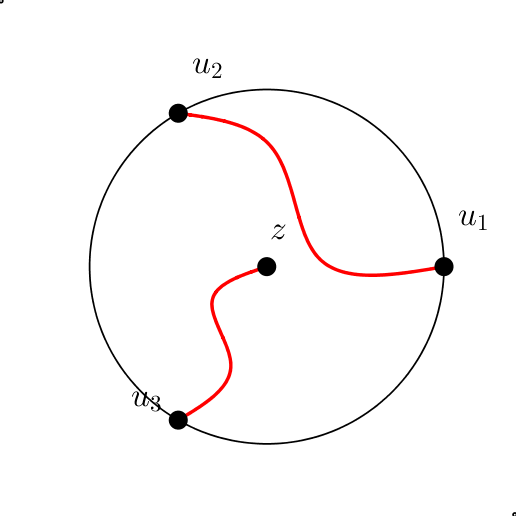}
    \caption{Link patterns $\CP{z}{u_1}$, $\CP{z}{u_2}$, $\CP{z}{u_3}$. 
}
    \label{fig:3+1}
\end{figure}

\begin{definition}
 We set $τ := \exp(2\pi i /3)$.  
 Let $u_1, u_2, u_3$ be three distinct mid-edges lying in the counterclockwise order on $∂ Ω$.
 By $\MidEdgesExcept (Ω)$ we denote the set of all mid-edges of $Ω$ {\it except for} $u_1, u_2, u_3$.
 Then the {\it observable} is a function $F  = F_{Ω, u_1, u_2, u_3}  \colon \MidEdgesExcept(Ω) \to \mathbb C$ given by the
formula
\begin{equation}
\label{eq:defF}
F(z) :=   \Eloop_{Ω, u_1, u_2, u_3, z} [H(ξ)]= \sum\limits_{j=1}^3  τ^j H_j(z),
\end{equation}
where 
$H(ξ) = \sum_{j=1}^3 τ^j \mathbf 1_{\CP {z}{u_j} }$ 
and 
$H_j(z) = \Ploop_{Ω, u_1, u_2, u_3, z} \CP {z}{u_j}  $.
\end{definition}

\begin{lemma}[Discrete holomorphicity]
\label{lemma:DiscreteHolomorphicity}
Let  $z_1, z_2, z_3 \in  \MidEdgesExcept(Ω)$ be three mid-edges around a vertex $v$ indexed in the counterclockwise order, then
\begin{equation} 
\sum_{k=1}^3 τ^k F(z_k) =0.
\label{eq:DiscreteHolomorphicity}
\end{equation}\end{lemma}
\begin{proof}

We group  loop configurations from $\cup_{z\in \{z_1, z_2, z_3\}} W(u_1, u_2, u_3, z)$ in triples such that 
any two loop configurations in the same triple differ by two half-edges adjacent to $v$ (See Fig. \ref{fig:DiscreteHolomorphicity}). 
Each triple contributes zero to 
$$\sum\limits_{k=1}^3 τ^k \sum_{ξ \in W_Ω(u_1, u_2, u_3, z_k)} H(ξ) .$$
\begin{figure}[ht!]
    \centering
    \includegraphics[width=0.75\textwidth]{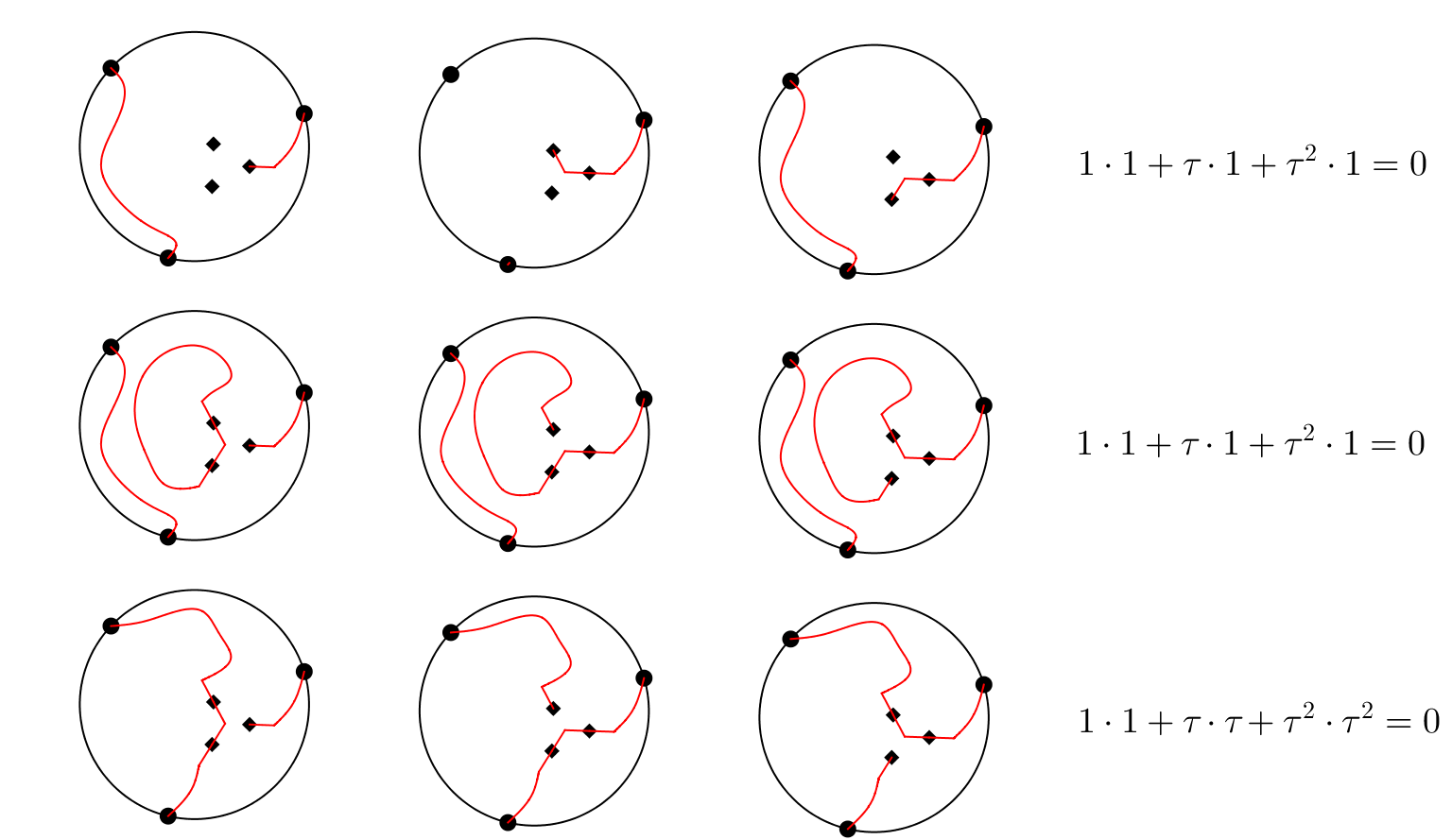}
    \caption{Graphical proof of Lemma \ref{lemma:DiscreteHolomorphicity}.
    Mid-edges  $z_1, z_2, z_3$ are marked with diamonds and  $u_1, u_2, u_3$ with circles.
     Configurations are grouped horizontally.}
    \label{fig:DiscreteHolomorphicity}
\end{figure}
\end{proof}

\begin{corollary}
Let $ γ$ be a dual contour, i.e. a sequence $(w_0, w_1, \dots w_n = w_0)$ 
of distinct faces where any two consecutive faces $w_j$ and $w_{j+1}$ share exactly one edge $e_j$. 
Then the discrete integral of $F$  along $γ$ defined by the formula
$$ \int\limits_{γ}^\# F(z)  \, d^{\#}z  := \sum\limits_{j=0}^{n-1} F(e_j) (w_{j+1}^\circ - w_{j}^\circ) $$
(here $w_{j}^\circ$ stands for the center of $w_j$) vanishes.
\end{corollary}
\begin{proof}
For an  elementary contour (i.e. that consists of three faces adjacent to the same vertex) the equality follows from
(\ref{eq:DiscreteHolomorphicity}).
Since any contour can be decomposed into a union of  elementary ones and the discrete integration is additive w.r.t contour,
the corollary is also true for arbitrary contour.
\end{proof}

\begin{rem}
The functions $H_1, H_2, H_3$ can also be defined on the vertices, though an interface can now arrive from three possible directions. Apparently, that would give exactly the same functions $H_1,
H_2, H_3$ as were defined in \cite{Smirnov01criticalpercolation} and $F$ as was defined in \cite{Beffara07EasyWay} under name $h$.
The Aizenman-Duplantier-Aharony recoloring \cite{AizenmanDuplantierAharony99} used  in \cite{Smirnov01criticalpercolation}
corresponds to the last triple in the Figure \ref{fig:DiscreteHolomorphicity}. 
\end{rem}

In terms of observable $F$  Lemma \ref{lemma:percloop}  says that if mid-edges $u_1, u_2, u_3$ lie on the boundary of $\Omega$
and a mid-edge $z$ lies on the boundary arc $∂_{j}\Omega$ then 
\begin{equation}
\label{eq:BV}
F(z) =  \Pperc_Ω [∂_{j+1,z}Ω ↔  ∂_{j-1, j}Ω] \cdot τ^{j-1} + 
        \Pperc_Ω [∂_{j, j+1}Ω ↔  ∂_{z, j-1}Ω] \cdot τ^{j+1} 
         \subset [τ^{j-1} , τ^{j+1}].
\end{equation}

\section{Theorem \ref{Thm1}  for the Jordan case } 
We denote by $\T$ the open domain bounded by the regular triangle with vertices $1, τ, τ^2$. 
For a simply-connected domain  $U$ with three chosen prime ends $A, B, C$ we denote by $φ_{U; A, B,C}$ the conformal map from
$U$ to $\T$ that maps $A$, $B$, $C$ to  $τ, τ^2, τ^3= 1$ respectively.

\begin{definition} \label{definition:NiceApproximation} Let $\LIM  {Ω} \subset \mathbb C$ be a bounded simply-connected domain
and $\LIM A$, $\LIM B$, $\LIM C$, $\LIM D$ be prime ends of $\LIM {Ω}$  lying in the counterclockwise order. Let $\{(Ω^δ, A^δ,
B^δ, C^δ, D^δ) \}_δ$ parametrized by $δ \searrow 0$ be a sequence such that $Ω^δ$ is a $\hexagon^δ$-domain, $A^δ, B^δ, C^δ,
D^δ$ are boundary mid-edges of $Ω^δ$. We say that  the sequence $(Ω^δ, A^δ, B^δ, C^δ, D^δ)$ {\it approaches} $(\LIM  {Ω}, \LIM
A, \LIM B, \LIM C, \LIM D)$ if Assumption \ref{assumption:JordanApproximation} or Assumption 
\ref{assumption:CaratheodoryApproximation} holds, see below.
\end{definition}

\begin{assumption} \label{assumption:JordanApproximation}
$∂ \LIM  {Ω}$ is a Jordan curve,  $Ω^δ$ is the $\hexagon^δ$-domain lying inside $\LIM  {Ω}$ of the maximal area and $A^δ, B^δ,
C^δ, D^δ$ are the boundary mid-edges of  $Ω^δ$ closest to $\LIM A$, $\LIM B$, $\LIM C$, $\LIM D$ respectively. 
\end{assumption}

\begin{proof}[Proof of Theorem \ref{Thm1}  under Assumption \ref{assumption:JordanApproximation}]
Let $F_δ$ be defined by  the formula (\ref{eq:defF}) for $(Ω, u_1, u_2, u_3) = (Ω^δ, A^δ, B^δ, C^δ)$. We denote by $f_δ$  the piecewise
linear extension of $F_δ$ defined as follows. First, define $f_δ$ on centers, mid-edges and vertices of all the hexagons
intersecting $Ω$ by $f_δ(u) := F_δ(u^δ)$, where $u^δ$ the mid-edge of $\MidEdgesExcept(Ω^δ)$ closest to $u$ (if there are
several closest mid-edges we choose one arbitrary). Then extend $f_δ$ linearly to each triangle spanned by adjacent vertex,
mid-edge and center of a face.

Lemma  \ref{lemma:DiscreteEstimates} implies that the family $\{f_δ\}_δ$ is uniformly Hölder on any  $K \Subset  \LIM {Ω}$.
Moreover, since $\LIM {Ω}$ is Jordan, it is locally connected: there exists $\zeta(\cdot) = o(1)$ near $0$ such that any two
points $x,y \in \LIM {Ω}$ can be joined inside $\LIM {Ω}$ by a curve of a diameter at most $\zeta(|x-y|)$. 
So from Lemma \ref{lemma:DiscreteEstimates} we can derive that the family $\{f_δ\}_δ$ is equicontinuous on $\overline{ \LIM
{Ω}}$.

By Arzel\`a--Ascoli theorem there is a continuous function $f \colon  \overline{ \LIM {Ω}} \to \mathbb C$ and a sequence 
$\{δ_n\}_n$ converging to 0 such that $f_{δ_n} \rightrightarrows f$ on $\overline{ \LIM {Ω}}$. Let  $γ \Subset  \LIM
{Ω}$ be any rectangular contour and  let $γ_{δ_n}$ be a dual contour of the maximal area lying inside $γ$. Then
$$\int\limits_γ f(z) \, dz
= \lim\limits_{n\to\infty} \int\limits_{γ_{δ_n}} f(z) \, dz  
= \lim\limits_{n\to\infty}
\int\limits_{γ_{δ_n}}^\# F_{δ_n}(z) \, d^{\#}z 
=0,$$
so $f$ is holomorphic by Morera's theorem. 

 From (\ref{eq:BV}) we conclude that $f$ maps $∂_{j} \LIM {Ω} $ to $∂_{j} \T =  [τ^{j+1}, τ^{j-1}]$. The argument
 principle implies that $f =  φ_{\LIM \Omega, \LIM A, \LIM B, \LIM C} =: φ $, so all subsequential limits of $\{f_δ\}$
 coincide. Again using (\ref{eq:BV}) we find that
\begin{equation*}
\lim\limits_{δ\searrow 0}\Pperc_{Ω^δ}[
∂_{A^δ B^δ} Ω^δ ↔  ∂_{C^δ D^δ} Ω^δ]
=
\lim\limits_{δ\searrow 0}
 \frac {φ(\LIM C)- F_δ(D^δ)}
{φ(\LIM C) - φ(\LIM A)}
= \frac {φ(\LIM C)- φ(\LIM D)}
{φ(\LIM C) - φ(\LIM A)}.
\end{equation*}
\end{proof}

\section{A priori estimates}
Our work requires only one non-trivial result on percolation: the famous Russo-Seymour-Welsh estimate.
We state it in the following way:

\begin{proposition}[RSW estimate] 
\label{prop:RSW'}
There exist $\eta > 0$ and $C_{\mathrm{RSW}}>0$ such that for any $r<R$ and for any $δ$
 $$\mathbb \Pperc_{\mathbb C_{\hexagon}^{δ}} 
 [∂ B_r ↔  ∂ B_R]  < C_{\mathrm{RSW}} (r/R) ^ \eta.$$
\end{proposition}

In order to make the proofs work for domains with possibly complicated boundaries, we define a metric on the closure $\overline
U$ of a Jordan domain $U$ by formula
$$ρ_U(x, y) := \inf\{\diam γ: γ \subset \overline { U} \text{ is a curve from } x \text{ to } y  \}$$
and formulate Lemma \ref{lemma:DiscreteEstimates} in terms of this metric.
To prove Theorem \ref{Thm1} for the case when $\LIM \Omega$ is smooth one can use the Euclidian metric instead of it.

\begin{lemma}[Hölder continuity]
\label{lemma:DiscreteEstimates}
There exist $\eta, C>0$ such that the following holds.
Let  $Ω$ be a $\hexagon^{δ}$-domain with three marked boundary mid-edges $v_1, v_2, v_3$. 
Assume that a set $S$ is such that $\overline \Omega \setminus S$ has a path-connected component, containing two mid-edges 
$x,y \in \MidEdgesExcept(\Omega)$ and at most one marked mid-edge, then 
\begin{equation}
\label{eq:DiscreteEstimates}
\forall j\ \  \ |H_j(x) - H_j(y)| <
 C \left( \frac{\diam S}{R} \right)^\eta,
\end{equation}
where $R = \max_k ρ_{\Omega}(S, ∂_k \Omega )$.
  \end{lemma}

\begin{proof} See Figure \ref{fig:proofofholder}.
We start by assuming that $R /100 > 100 \diam S > \delta$, otherwise Lemma follows by chosing large enough $C$.
Without loss of generality $R = ρ_{\Omega}(S, ∂_3 \Omega)$, so $v_1$, $v_2$ are outside of the path connected component of 
$\overline \Omega \setminus S$  that contains  $x, y$.
Let $\tilde S$ be the $(10\delta)$-neighborhood of $S$  with respect to $ρ_\Omega$.
We choose a $\hexagon$-path $[xy]$ such that no path joining $[xy]$ and $\{v_1, v_2\}$ is disjoint from $\tilde S$.
Clearly, the LHS of (\ref{eq:DiscreteEstimates}) is bounded by $\Ploop_{Ω, v_1, v_2, v_3, x} [ H(ξ)\ne  H(ξ \oplus [xy])]$,
and let us call configurations $ξ$ such that the last event occurs {\it bad} and denote the set of bad configurations by
$W^{\mathrm{bad}}$.

If $ξ \in W^{\mathrm{bad}}$, then $[xy]$ should be connected to each of $v_1, v_2, v_3$ by edges of $ξ$; so $\tilde S$ is
connected to $v_1$ and $v_2$ by edges of $ξ$. 
Let $β(ξ)$ be the minimal subset of $ξ$ that connects $\tilde S$ to $v_1$ and $v_2$ (this is a union of two paths). 
Now note that there is a path in $\HalfEdges(\Omega) \setminus β(ξ)$ from $x$ to $u_3$. For each $ξ \in W^{\mathrm{bad}}$
we choose such path $α(β(ξ))$ in any way depending on $β(ξ)$ but not on $ξ$ itself.

\begin{figure}[ht!]
\begin{center}
\includegraphics[width = 0.8\textwidth]{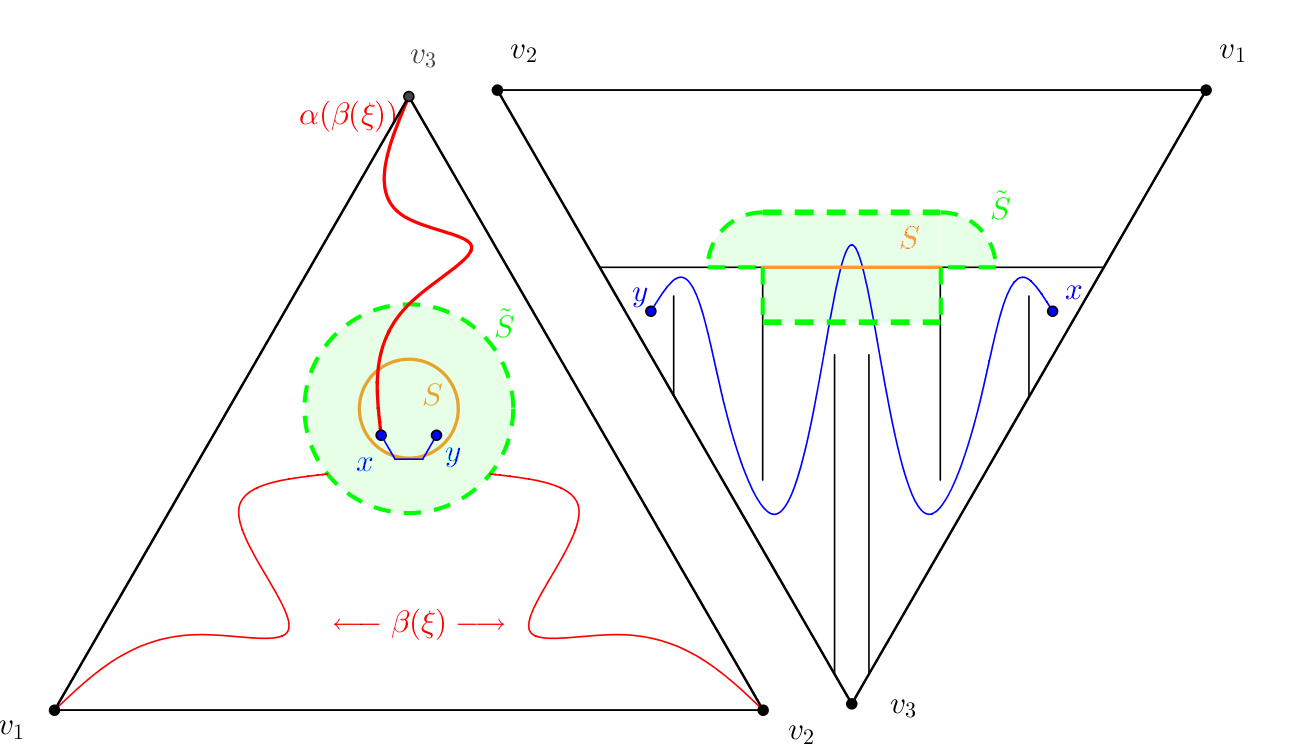}
\caption{Two examples of a $\hexagon$-path $[xy]$ (in blue), sets $S$ (in orange) and $\tilde S$ (in green)
and  $\hexagon$-paths $α(β(ξ)), β(ξ)$ (in red, only on the left).
 In general, $S$ either surrounds $x$ and $y$ as on the left, 
or cuts away the part of $\Omega$ containing them, as on the right. On the right we sketch how such cuts would look if $\LIM \Omega$ is non-Jordan.}
\label{fig:proofofholder}
\end{center}
\end{figure}

Note that the map $ξ \mapsto ξ \oplus α(β(ξ))$ is injective on $W^{\mathrm{bad}}$.
Moreover, if $ξ \in W^{\mathrm{bad}}$, then $ξ \oplus α(β(ξ)) \oplus  ∂_3 \Omega$ is a loop configuration without
disorders that contains a loop touching $\tilde S$ and $∂_3 \Omega$.
Then the corresponding coloring defined as in the proof of Lemma \ref{lemma:percloop} (assuming that the outer boundary of
$\Omega$ is yellow) contains a  monochromatic path between $\tilde S$ and $∂_3 \Omega$.
Since the diameter of any such path is at least $R/2$, we conclude by the RSW estimate.
\end{proof}

\section{Theorem \ref{Thm1} for the general case } 
In this section we work under the following assumption, which is more general than Assumption \ref{assumption:JordanApproximation}.

\begin{assumption}[convergence in the Carathéodory sense]  \label{assumption:CaratheodoryApproximation}
$\LIM {Ω}$ is an arbitrary bounded simply-connected domain and the following properties hold:  
\begin{itemize}
 \item Any $K$ such that $K \Subset {\LIM {Ω}} $ is contained in $Ω^δ$ for $δ$ small enough;
 \item $φ_{Ω^δ; A^δ, B^δ,C^δ}^{-1} =: φ_{δ}^{-1}$ converges to  $φ_{\LIM  {Ω}; \LIM A, \LIM B, \LIM C}^{-1} =: φ^{-1} $  uniformly on any compact $K\Subset \T$;
 \item $φ_{Ω^δ; A^δ, B^δ,C^δ} (D^δ) $ converges to  $φ_{\LIM {Ω} ; \LIM A, \LIM B, \LIM C}(\LIM D)$;
 \item $\cup_δ Ω^δ$ is bounded.
\end{itemize} 
\end{assumption}

\begin{proof}[Proof of Theorem \ref{Thm1}  for the general case]
As in the proof for the Jordan case, using Lemma \ref{lemma:DiscreteEstimates}  we define
functions $f_δ$ and find a sequence  $\{δ_n\}_n$ converging to 0 and a holomorphic
function  $f$ on $\LIM {Ω}$ such that
\begin{equation} 
f_{δ_n} \rightrightarrows f \text{ on any } K \Subset  \LIM {Ω}.
\end{equation}

To analyze its boundary behavior in the prime end (`Carathéodory') compactification, we extend $ φ_{δ_n}^{-1}$ to
$\overline{\T}$ by continuity and note that the sequence $f_{δ_n} \circ φ_{δ_n}^{-1} \colon \overline{\T} \to \mathbb C$
uniformly converges to $f \circ φ^{-1}$ on any compact subset of $\T$. 
Then we aim to show that this sequence is equicontinuous on $\overline{\T}$.

For $x, y \in \overline {\T}$ and any $δ$ we consider the set of simple (possibly closed) curves $γ \subset \overline{Ω^{δ}}$ 
such that $\overline{Ω^{δ}} \setminus γ$ consist of exactly two path-connected components, one containing $φ_{δ}^{-1}(x),
φ_{δ}^{-1}(y)$ and another one containing at least two of marked points $A^{δ}, B^{δ}, C^{δ}$ and denote by $\tilde {ρ}_δ(x,
y)$ the infimum of lengths of those curves.
By estimating the extremal length one can easily show that
\begin{equation} \label{eq:rhoestimate}
\lim_{\epsilon \searrow 0} \sup_δ \sup_{x, y \in \overline {\T}, |x-y|<\epsilon} \tilde  {ρ}_δ(x, y) = 0.
\end{equation}

Now we note that the family $\{Ω^δ\}_δ$ is non-degenerate in the following sense:
\begin{equation} \label{InRadiusBound}
\LIM r:=\liminf_{δ \searrow 0}	\inf_{x \in Ω^δ} \max_k ρ_{Ω^δ}(x, ∂_k Ω^δ) > 0.
\end{equation}

Indeed, assume the contrary, then for any $δ$ there exists a  path-connected  $Y_δ \subset \overline {Ω^δ}$ touching all three
boundary arcs of ${Ω^δ}$ such that $\liminf \diam Y_δ = 0$ as ${δ\to 0}$.
Let $O$ be the center of $\T$ and set $O_δ := φ_δ^{-1} (O)$.
Since $φ_δ^{-1}$ uniformly converges to $φ^{-1}$ on some open neighborhood of $O$, the distance from $O_δ $ to $∂ Ω^δ$ is
bounded from below, so  $\liminf \dist(O_δ ,  Y_δ ) > 0$. At the same time the harmonic measures with the pole at $O_δ $ of
arcs $∂_{A^δ, B^δ} Ω^δ$, $∂_{B^δ,C^δ} Ω^δ$, $∂_{C^δ, A^δ} Ω^δ$ equal to $1/3$.
Since one of those arcs is separated from $O_δ $ by $Y_δ$, (\ref{InRadiusBound}) is proven by contradiction.

Now for $x, y \in \overline {\T}$ at a small distance, for any $δ$ we can disconnect $φ_{δ}^{-1}(x), φ_{δ}^{-1}(y)$ from at
least two marked points by a curve $γ$ of small diameter by (\ref{eq:rhoestimate}).
Then we plug $S = γ$ in (\ref{eq:DiscreteEstimates}) and estimate the denominator in the RHS of (\ref{eq:DiscreteEstimates}) by
the triangle inequality $\max_k ρ_{Ω^δ}(γ, ∂_k Ω^δ) \geq\inf_{x \in Ω^δ} \max_k ρ_{Ω^δ}(x, ∂_k Ω^δ) - \diam γ$ and 
(\ref{InRadiusBound}). 
From that we conclude  the sequence $f_{δ_n} \circ φ_{δ_n}^{-1} $ is equicontinuous on $\overline{\T}$. 
As in the proof for the Jordan case, it follows from (\ref{eq:BV}) that  $f\circ φ^{-1}$ maps $∂_{j} \T =  [τ^{j+1}, τ^{j-1}]$
to itself, so by the argument principle
$$ f_δ \circ φ_δ^{-1}\rightrightarrows f\circ φ^{-1} = id \text{ on }  \overline \T$$
which in turn implies that
\begin{equation*}
 \lim\limits_{δ\searrow 0}   \Pperc_{Ω^δ}  [∂_{A^δ B^δ} Ω^δ ↔ ∂_{C^δ D^δ} Ω^δ]
=\lim\limits_{δ\searrow 0} 
\frac 
	{φ(\LIM C)- (f_δ \circ φ_δ^{-1})\circ φ_δ(D^δ)}
	{φ(\LIM C) - φ(\LIM A)}
= 
\frac 
	{φ(\LIM C)- φ(\LIM D)}
	{φ(\LIM C) - φ(\LIM A)}.
\end{equation*}

\end{proof}

\phantom
{\cite{Beffara07EasyWay ,beffara2013critical, Jiang17, smirnov2009critical, Kesten82, Grimmett99, Russo78, SeymourWelsh78, Pommerenke92, 
KhristoforovPHD}}

\bibliography{bb}{}

\begin{thebibliography}{10}

\bibitem{AizenmanDuplantierAharony99}
M.~Aizenman, B.~Duplantier, and A.~Aharony.
\newblock Path-crossing exponents and the external perimeter in 2d percolation.
\newblock {\em Phys. Rev. Lett.}, 83:1359--1362, Aug 1999.

\bibitem{Beffara07EasyWay}
V.~Beffara.
\newblock Cardy's formula on the triangular lattice, the easy way.
\newblock In {\em Universality and renormalization}, volume~50 of {\em Fields
  Inst. Commun.}, pages 39--45. Amer. Math. Soc., Providence, RI, 2007.

\bibitem{beffara2013critical}
V.~Beffara.
\newblock Critical percolation on mesoscopic triangulations, 2013.

\bibitem{BroadbentHammersley57}
S.~R. Broadbent and J.~M. Hammersley.
\newblock Percolation processes. {I}. {C}rystals and mazes.
\newblock {\em Proc. Cambridge Philos. Soc.}, 53:629--641, 1957.

\bibitem{Cardy92}
J.~L. Cardy.
\newblock Critical percolation in finite geometries.
\newblock {\em J. Phys. A}, 25(4):L201--L206, 1992.

\bibitem{CD-CHKS}
D.~Chelkak, H.~Duminil-Copin, C.~Hongler, A.~Kemppainen, and S.~Smirnov.
\newblock Convergence of {Ising} interfaces to {Schramm}'s {SLE} curves.
\newblock {\em Comptes Rendus Mathematique}, 352(2):157 -- 161, 2014.

\bibitem{ChelkakSmirnov12}
D.~Chelkak and S.~Smirnov.
\newblock Universality in the 2{D} {I}sing model and conformal invariance of
  fermionic observables.
\newblock {\em Invent. Math.}, 189(3):515--580, 2012.

\bibitem{Hugo13Parafermionic}
H.~Duminil-Copin.
\newblock {\em Parafermionic observables and their applications to planar
  statistical physics models}, volume~25 of {\em Ensaios Matem\'{a}ticos
  [Mathematical Surveys]}.
\newblock Sociedade Brasileira de Matem\'{a}tica, Rio de Janeiro, 2013.

\bibitem{FloresSimmonsKleban15}
S.~M. Flores, J.~J. Simmons, and P.~Kleban.
\newblock Multiple-sle connectivity weights for rectangles, hexagons, and
  octagons.
\newblock {\em arXiv preprint arXiv:1505.07756}, 2015.

\bibitem{Grimmett99}
G.~Grimmett.
\newblock {\em Percolation}, volume 321 of {\em Grundlehren der Mathematischen
  Wissenschaften [Fundamental Principles of Mathematical Sciences]}.
\newblock Springer-Verlag, Berlin, second edition, 1999.

\bibitem{HonglerSmirnov13}
C.~Hongler and S.~Smirnov.
\newblock The energy density in the planar {I}sing model.
\newblock {\em Acta Math.}, 211(2):191--225, 2013.

\bibitem{Jiang17}
J.~Jiang.
\newblock Exploration processes and {${\rm SLE}_6$}.
\newblock {\em Markov Process. Related Fields}, 23(3):445--465, 2017.

\bibitem{Kesten82}
H.~Kesten.
\newblock {\em Percolation theory for mathematicians}, volume~2 of {\em
  Progress in Probability and Statistics}.
\newblock Birkh\"{a}user, Boston, Mass., 1982.

\bibitem{KhristoforovPHD}
M.~Khristoforov.
\newblock {\em Low dimensional defects in percolation model}.
\newblock PhD thesis, 2018.
\newblock ID: unige:111513.

\bibitem{tmpKK}
M.~Khristoforov and V.~Kleptsyn.
\newblock {Smirnov}’s observable and link pattern probabilities in
  percolation.
\newblock {\em In preparation}.

\bibitem{tmpKSS}
M.~Khristoforov, M.~Skopenkov, and S.~Smirnov.
\newblock A generalization of {Cardy}’s and {Schramm}’s formulae.
\newblock {\em In preparation}.

\bibitem{LanglandsEtAl94}
R.~Langlands, P.~Pouliot, and Y.~Saint-Aubin.
\newblock Conformal invariance in two-dimensional percolation.
\newblock {\em Bull. Amer. Math. Soc. (N.S.)}, 30(1):1--61, 1994.

\bibitem{LawlerSchrammWerner04}
G.~F. Lawler, O.~Schramm, and W.~Werner.
\newblock Conformal invariance of planar loop-erased random walks and uniform
  spanning trees.
\newblock {\em Ann. Probab.}, 32(1B):939--995, 2004.

\bibitem{Pommerenke92}
C.~Pommerenke.
\newblock {\em Boundary behaviour of conformal maps}, volume 299 of {\em
  Grundlehren der Mathematischen Wissenschaften [Fundamental Principles of
  Mathematical Sciences]}.
\newblock Springer-Verlag, Berlin, 1992.

\bibitem{Russo78}
L.~Russo.
\newblock A note on percolation.
\newblock {\em Z. Wahrscheinlichkeitstheorie und Verw. Gebiete}, 43(1):39--48,
  1978.

\bibitem{Schramm01PercolationFormula}
O.~Schramm.
\newblock A percolation formula.
\newblock {\em Electron. Comm. Probab.}, 6:115--120, 2001.

\bibitem{SeymourWelsh78}
P.~D. Seymour and D.~J.~A. Welsh.
\newblock Percolation probabilities on the square lattice.
\newblock {\em Ann. Discrete Math.}, 3:227--245, 1978.
\newblock Advances in graph theory (Cambridge Combinatorial Conf., Trinity
  College, Cambridge, 1977).

\bibitem{Smirnov01criticalpercolation}
S.~Smirnov.
\newblock Critical percolation in the plane: conformal invariance, {C}ardy's
  formula, scaling limits.
\newblock {\em C. R. Acad. Sci. Paris S\'{e}r. I Math.}, 333(3):239--244, 2001.

\bibitem{smirnov2009critical}
S.~Smirnov.
\newblock Critical percolation in the plane, 2009.

\bibitem{Smirnov10}
S.~Smirnov.
\newblock Conformal invariance in random cluster models. {I}. {H}olomorphic
  fermions in the {I}sing model.
\newblock {\em Ann. of Math. (2)}, 172(2):1435--1467, 2010.

\end{thebibliography}
\bibliographystyle{abbrv} 
\end{document}